\documentclass[11pt]{article}
\usepackage{amsmath,amsthm,amsfonts,amssymb,amscd, amsxtra, mathrsfs}
\usepackage{url}
\usepackage[margin=2.5cm,nohead]{geometry}
\usepackage{color}
\newcommand{\banacha}{X}
\newcommand{\banachb}{Y}

\newcommand{\banachh}{H}
\newtheorem{theorem}{Theorem}
\newtheorem{lemma}[theorem]{Lemma}
\newtheorem{definition}{Definition}
\newtheorem{corollary}[theorem]{Corollary}

\newtheorem{remark}{Remark}
\begin{document}

\title{Kantorovich's theorem on Newton's Method for solving generalized equations under the majorant condition}

\author{Gilson N. Silva \thanks{CCET/UFOB,  CEP 47808-021 - Barreiras, BA, Brazil (Email: {\tt  gilson.silva@ufob.edu.br}). The author was supported in part by CAPES .}    
}
\date{March 14, 2016}

\maketitle
\maketitle
\begin{abstract}
In this paper we consider a version of the Kantorovich's theorem for solving the generalized equation $F(x)+T(x)\ni 0$, where $F$ is a Fr\'echet derivative function and $T$ is a set-valued and maximal monotone acting between Hilbert spaces. We show that this method is  quadratically convergent to a solution of $F(x)+T(x)\ni 0$. We have used the idea of majorant function, which relaxes  the Lipschitz continuity of the derivative $F'$. It allows us to obtain the optimal convergence radius, uniqueness of solution and also to solving generalized equations under Smale's condition.
\end{abstract}
\medskip
 

\maketitle

\section{Introduction}\label{sec:int}
In this paper we are interested in the solution of the generalized equation 
\begin{equation} \label{eq:ipi}
  F(x) +T(x) \ni 0, 
\end{equation}
where $F:{\Omega}\to \banachh$ is a Fr\'echet differentiable function, $\banachh$ is a Hilbert space, $\Omega\subseteq \banachh$ an open set and $T:\banachh \rightrightarrows  \banachh$ is set-valued and maximal monotone. As is well known, the generalized equation \eqref{eq:ipi} covers wide range of problems in classical analysis and its applications.  For instance, systems of nonlinear equations and abstract inequality systems. If $\psi: \banachh \to (-\infty, +\infty]$ is a proper lower semicontinuous convex function and 
$$
T(x)=\partial \psi(x) =\{u\in \banachh ~:~ \psi(y) \geq \psi(x) + \langle u,y-x\rangle\}, \qquad \forall~ y\in \banachh,
$$
then \eqref{eq:ipi} becomes the variational inequality problem 
$$
F(x) +\partial \psi(x) \ni 0,
$$
including linear and nonlinear complementary problems; additional comments about such problems can be found in \cite{ DontchevRockafellar2009, FerreiraSilva, josephy1979, Robinson1972_2, Uko1996, UkoArgyros2009}, and the references cited therein. 

Newton's method is one of the most important methods to solve \eqref{eq:ipi} which generates a sequence,  for  an  initial point $x_0$,  as follows
\begin{equation} \label{eq:ipi1}
  F(x_k) + F'(x_k)(x_{k+1}-x_k)+ T(x_{k+1}) \ni 0, \qquad k=0,1, \ldots.
\end{equation}
This method may be viewed as a Newton-type method based on a partial linearization, which has been studied in several papers including \cite{ Dontchev1996, PietrusJean2013};  see also \cite[ Section 6C]{DontchevRockafellar2009} where a interesting discussion about  iterative methods for solving generalized equations is presented. When $F\equiv 0$, the iteration \eqref{eq:ipi1} becomes the standard Newton's method for solving the nonlinear equation $F(x)=0.$  

L.~V. Kantorovich in \cite{Kantorovich1948}, established a convergence result to a solution of the Newton method for solving the equation $F(x)=0$ using conditions where the iterations begin. To be more specific, by imposing assumptions on the derivative $F'(x_0)$ and of the term $\|F'(x_0)^{-1}F(x_0)\|$ he obtained a convergence result to a solution of the Newton method. The main idea of Kantorovich, used in his proof, was the majorization principle of the sequence $x_k$ by a sequence of scalars. Recently, there has been some growing interest in Kantorovich's theorem, see for instance \cite{Ferreira2015, FerreiraSilva, FerreiraSvaiter2009, Potra2005, Wang1999}.

Up to our knowledge, S.~M. Robinson in \cite{Robinson1972_2}, was the first to consider a generalization of the Kantorovich theorem of the type $F(x)\in K$, where $K$ is a nonempty closed and convex cone, giving convergence results and error bounds for this method.  His idea was to use properties of convex process introduced by Rockafellar \cite{Rockafellar1970}, for more details and results about convex process see \cite{Robinson1972_1, Rockafellar1967}. In \cite{Robinson1972_2} the iterative $x_{k+1}$ is taken as the element of the set of points $x_{k+1}$ defined by \eqref{eq:ipi1} which is at minimal distance from $x_k$. Thus is obtained that the sequence $x_k$ converges quadratically to a solution of $F(x)\in K.$  

N.~H. Josephy in \cite{josephy1979} was the first to consider a semi local Newton's method of the kind \eqref{eq:ipi1},  in  order to solving (\ref{eq:ipi}),  to $F=N_C$ the normal cone mapping of a convex set $C\subset \mathbb{R}^{m}$.  For guarantee  the well definition of the method, in the theory of generalized equations, the property of {\it strong regularity} of $F(x_0)+F'(x_0)(x-x_0)+T(x)$ at $x_1$ for $0$, where $x_1$ is obtained from $x_0$, introduced by S.~M. Robinson in \cite{Robinson1980}, was used.   If $\banacha = \banachb$ and $F=\{ 0\},$ then strong regularity of $f(x_0)+f'(x_0)(x-x_0)+F(x)$ at $x_1\in \banacha$  for $0\in  \banacha$ is equivalent to assumption that   $f'(x_0)^{-1}$  is a continuous linear operator.  If  $\banacha=\mathbb{R}^{n}$, $\banachb=\mathbb{R}^{m}$ and $F=\mathbb{R}^{s}_{-}\times \{0\}^{m-s}$, then  strong regularity is equivalent to   Mangasarian-Fromovitz constraint qualification; see \cite[Example 4D.3]{DontchevRockafellar2009}.  An important case is when \eqref{eq:ipi} represents the  Karush-Kuhn-Tucker's  systems for  the standard nonlinear programming problem with a strict local minimizer, see  \cite{DontchevRockafellar2009} pag. 232. 

In \cite{chang2015, UkoArgyros2009} under a majorant condition and generalized Lipschitz condition, local and semi local convergence, quadratic rate and estimate of the best possible convergence radius of Newton's method as well as uniqueness of the solution for solving generalized equation were established. One usual assumption used to obtain quadratic convergence of Newton's method \eqref{eq:ipi1},   for solving equation \eqref{eq:ipi},  is  the Lipschitz continuity of  $F'$  in a neighborhood  the solution; see  \cite{DontchevRockafellar2009}.  Indeed,   keeping control of the derivative is an important point in the convergence analysis of Newton's method.  On the other hand, a couple of papers have dealt with the issue of convergence analysis of the Newton's method,  for solving  the equation $F(x)=0$,   by relaxing the assumption of Lipschitz continuity of $F'$, see for example \cite{FerreiraSvaiter2009, Wang1999, Zabrejko1987}, actually all this conditions 
are equivalent to X. Wang's condition  introduced in \cite{Wang1999}. The advantage of working with a majorant condition rests in the fact that it allow to unify several convergence results pertaining to  Newton's method; see  \cite{FerreiraSvaiter2009, Wang1999}; see Section~\ref{apl}.

In this paper we rephrase  the majorant  condition introduced in \cite{FerreiraSvaiter2009}, in order to study the local convergence properties of Newton's method \eqref{eq:ipi1}. The analysis presented provides a clear relationship between the majorant function and the function defining the generalized equation. Also, it allows us to obtain the optimal convergence radius for  the method with respect to the majorant condition and uniqueness of solution.  The analysis of this method, under Lipschitz's condition and Smale's condition,  are provided  as special case.

The organization of the paper is as follows. In Section~\ref{sec:int.1},  some notations and important results  used  throughout  the paper are presented. In Section \ref{lkant}, the main result is stated and  in Section~\ref{sec:PMF} properties of the majorant function,  the main  relationships   between the majorant function and the nonlinear operator are established. In  Section~\ref{convanalysis} the main result is proved, the uniqueness of the solution and some applications of this result are given in Section~\ref{apl}. Some final remarks are made in  Section~\ref{rf}.
\section{Preliminaries} \label{sec:int.1}
The following notations and results are used throughout our presentation. Let $\banachh$ be a Hilbert space with scalar product $\langle ., .\rangle$  and norm $\|.\|$, the {\it open} and {\it closed balls} at $x$ with radius $\delta\geq 0$ are denoted, respectively, by $ B(x,\delta) := \{ y\in X ~: ~\|x-y\|<\delta \}$ and  $B[x,\delta] := \{ y\in X ~: ~\|x-y\|\leqslant \delta\}.$ 

We denote by ${\mathscr L}(\banacha,\banachb)$ the {\it space consisting of all continuous linear mappings} $A:\banacha \to \banachb$ and   the {\it operator norm}  of $A$ is defined  by $  \|A\|:=\sup \; \{ \|A x\|~:  \|x\| \leqslant 1 \}.$ Recall that a bounded linear operator $G:\banachh \to \banachh$ is called a positive operator if $G$ is a self-conjugate and $\langle Gx,x\rangle \leq 0$ for each $x\in \banachh$. The {\it domain} and the {\it range} of $H$ are, respectively,  the sets $ \mbox{dom}~H:=\{x\in \banacha ~: ~ H(x)\neq \varnothing\} $ and $ \mbox{rge}~H:=\{y\in \banachb ~: ~ y \in H(x) ~\emph{for some} ~x\in \banacha\}$. The {\it inverse} of $H$  is the set-valued mapping  $H^{-1}:\banachb \rightrightarrows  \banacha$ defined by $ H^{-1}(y):=\{x \in \banacha ~: ~ y \in H(x)\}$.

\begin{definition}\label{def:pplm}
Let $\banachh$, be a Hilbert space, $\Omega$ be an open nonempty subset of $\banachh$,  $h: \Omega \to \banachh$ be a Fr\'echet derivative  with derivative  $h'$ and $T:\banachh \rightrightarrows  \banachh$ be a set-valued mapping. The {\it partial linearization} of the mapping   $h +T$ at $x\in \banachh$  is   the set-valued mapping   $L_h(x, \cdot ):\banachh \rightrightarrows  \banachh$ given by 
\begin{equation} \label{eq:pplm}
L_h(x, y ):=h(x)+h'(x)(y-x)+T(y).
\end{equation}
For each $x\in \banachh$, the inverse   $L_h(x, \cdot )^{-1}:\banachh \rightrightarrows  \banachh$ of the  mapping $L_h(x, \cdot )$  at $z\in \banachh$ is denoted  by
\begin{equation} \label{eq:invplm}
L_h(x, z )^{-1}:=\left\{y\in X ~:~ z\in h(x)+h'(x)(y-x)+T(y)\right\}.
\end{equation}
\end{definition}
\begin{remark}
If in above definition we have $T\equiv {0}$, $z=0$ and $h'(x)$ invertible, then $L_h(x, 0 )^{-1}=x-h'(x)^{-1}h(x)$ is the Newton iteration for solving the equation $h(x)=0$. 
\end{remark}

Now, we recall notions of monotonicity for set-valued operators.
\begin{definition}\label{def.mono}
Let $T:\banachh \rightrightarrows  \banachh$ be a set-valued operator. $T$ is said to be monotone if for any $x,y\in \mbox{dom}~{T}$ and, $u \in T(y)$, $v\in T(x)$ implies that the follwing inequality holds:
$$
\langle u-v,y-x \rangle \geq 0 .
$$
\end{definition}
A subset of $\banachh \times \banachh$ is monotone if it is the graph of a monotone operator. If $ \varphi: \banachh \to (-\infty, +\infty]$ is a proper function then the subdifferential of $\varphi$ is  monotone. 
\begin{definition}
Let $T:\banachh \rightrightarrows  \banachh$ be monotone. Then $T$ is maximal monotone if the following implication holds for all $x,u\in \banachh$:
\begin{equation}
\langle u-v,y-x \rangle \geq 0 \quad \mbox{for each} \quad y\in \emph{dom}{T} \quad \mbox{and}\quad v\in T(y) \Rightarrow \quad x\in \emph{dom}{T} \quad \emph{and} \quad v\in T(x).
\end{equation}
\end{definition}
An example of maximal monotone operator is the subdifferential of a proper, lower semicontinuous, convex function $ \varphi: \banachh \to (-\infty, +\infty]$.
The following result can de found in \cite{Wang2015}.
\begin{lemma}\label{eq:plm}
Let $G$ be a positive operator. The following statements about $G$ hold:
\begin{enumerate}
\item $\|G^2\|=\|G\|^2$;
\item If $G^{-1}$ exists,  then $G^{-1}$ is a positive operator.
\end{enumerate}
\end{lemma}
As a consequence of this result we have the following result:
\begin{lemma}\label{eq:adjunt}
Let $G$ be a positive operator. Suppose that $G^{-1}$ exists, then for each $x\in \banachh$ we have
$$
\langle Gx,x\rangle \geq \frac{\|x\|^2}{\|G^{-1}\|}.
$$
\end{lemma}
\begin{proof}
See Lemma~2.2 of \cite{Uko1996}.
\end{proof}

Let $G:\banachh \to \banachh$ be a bounded linear operator. We will use the convention that $\widehat{G}:=\frac{1}{2}(G+G^*)$ where $G^*$ is the conjugate operator of $G$. As we can see, $\widehat{G}$ is a self-conjugate operator.
From now, we assume that $T:\banachh \rightrightarrows  \banachh$ is a set-valued  maximal monotone operator and $F: \banachh \to \banachh$ is a Fr\'echet derivative function. The next result is of major importance to prove the good definition of Newton's method. Its proof can be found in \cite[Lemma~1, p.189]{Smale1986}.
\begin{lemma}[Banach's lemma]\label{eq.banachlemma}
Let $B: \banachh \to \banachh$ be a bounded linear operator and $I:\banachh \to \banachh$ the identity operator. If $\|B-I\|<1$ then $B$ is invertible and $\|B^{-1}\|\leq 1/(1-\|B-I\|)$.
\end{lemma}
\section{Local analysis of Newton's method } \label{lkant}
In this section,  we study the Newton's method for solving the generalized equation \eqref{eq:ipi},  which is based in the partial linearization of this equation, see \cite{josephy1979} (see also, \cite{Dontchev1996}).  For study the convergence properties of this method, we assume that  the derivative $F'$ satisfies a weak Lipschitz condition on a region $\Omega$  relaxing  the usual Lipschitz condition. The statement of  the our main result is:
\begin{theorem}\label{th:nt}
Let $\banachh$ be a Hilbert space, $\Omega$ be an open nonempty subset of $\banachh$, $F: \Omega \to \banachh$ be continuous with Fr\'echet derivative $F'$ continuous, $T:\banachh \rightrightarrows  \banachh$ be a set-valued operator and $x_0\in \Omega$. Suppose that $F'(x_0)$ is a positive operator and $\widehat{F'(x_0)}^{-1}$ exists. Let $R>0$  and  $\kappa:=\sup\{t\in [0, R): B(x_0, t)\subset \Omega\}$. Suppose that there exist $f:[0,\; R)\to \mathbb{R}$ twice continuously differentiable such that  
  \begin{equation}\label{Hyp:MH}
\|\widehat{F'(x_0)}^{-1}\| \left\|F'(y)-F'(x)\right\| \leq f'\left(\|x-y\|+\|x-x_0\| \right)-f'\left(\|x-x_0\|\right),
  \end{equation}
  for all $x,y \in B(x_0, \kappa)$. Moreover, suppose that
	\begin{equation} \label{eq.ipoint1}
	\|x_1-x_0\|\leq f(0),
	\end{equation}
	and the following conditions hold,
	\begin{itemize}
  \item[{\bf h1)}]  $f(0)>0$ and $f'(0)=-1$;
  \item[{\bf  h2)}]  $f'$ is convex and strictly increasing.
	\item [{\bf  h3)}] $f(t)=0$ for some $t\in (0,\; R)$.
\end{itemize}
Then, $f$ has a smallest zero $t^*\in (0,\; R)$, the sequences generated by Newton's method for solving the generalized equation $F(x)+T(x)\ni 0$ and the equation  $f(t)=0$, with starting point $x_0$ and $t_0=0$, respectively,
\begin{equation} \label{eq:DNS}
 0\in F(x_k)+F'(x_k)(x_{k+1}-x_k)+T(x_{k+1}), \qquad t_{k+1} ={t_k}-f(t_k)/f'(t_k),\qquad k=0,1,\ldots\,, 
\end{equation}
are well defined, $\{t_k\}$ is strictly increasing, is contained in $(0, t^*)$ and converges to $t^*$, $\{x_k\}$ is contained in $B(x_0, t^*)$ and  converges to the point $x^*\in B[x_0, t^*]$ which is the unique solution of the generalized equation $F(x)+T(x)\ni 0$ in $B[x_0, t^*]$.  Moreover, the sequences $\{x_k\}$ and $\{t_k\}$ satisfies,
\begin{equation}\label{eq:q2}
   \|x_*-x_k\| \leq t_* -t_k, \qquad \qquad   \|x_*-x_{k+1}\| \leq \frac{t_*-t_{k+1}}{(t_* -t_k)^2}\|x_*-x_k\|^2,  
  \end{equation}
for all k=0,1,..., and the sequences $\{t_k\}$ and $\{x_k\}$ converge $Q$-linearly as follows
\begin{equation}\label{eq:rates0}
\|x_*-x_{k+1}\| \leq \frac{1}{2}\|x_* -x_k\|, \qquad \qquad  t_* -t_{k+1} \leq \frac{1}{2} (t_* -t_k) \qquad k=0,1, \ldots\ .
\end{equation}
If, additionally
\begin{itemize}
  \item[{\bf h4)}] $f'(t_*)<0$,
	\end{itemize}
	then the sequences, $\{t_k\}$ and $\{x_k\}$ converge $Q$-quadratically as follows
	\begin{equation}\label{ine.rates1}
	\|x_*-x_{k+1}\| \leq \frac{D^{-}f'(t_*)}{-2f'(t_*)}\|x_*-x_k\|^2, \qquad \qquad  t_* -t_{k+1} \leq \frac{D^{-}f'(t_*)}{-2f'(t_*)} (t_* -t_k)^2,
	\end{equation}
	for all $k=0,1, \ldots\ .$
\end{theorem}
\begin{remark}
When $F\equiv {0}$ and $f'$ satisfies a Lipschitz-type condition, we will obtain a particular instance of Theorem~\ref{th:nt}, which retrieves the classical convergence theorem on Newton's method under the Lipschitz condition; see \cite{Rall1974, Traub1979}. 
\end{remark}
\begin{remark}\label{def.good}
Since $T$ is monotone maximal, if there exists a constant $c>0$ such that 
\begin{equation}\label{eq.gooddef}
\langle F'(x_k)y,y\rangle \geq c\|y\|^2
\end{equation}
for each $y\in \banachh$, then there exist an unique point $x_{k+1}$ such that the first inclusion in \eqref{eq:DNS} holds. The proof of this result can be found in \cite[Lemma~2.2]{Uko1996}. Hence, if for each $k$, there exist a constant $c>0$ such that \eqref{eq.gooddef} holds, then the sequence generated by \eqref{eq:DNS} is well defined.
\end{remark}
{\it From now on, we  assume that the hypotheses of Theorem \ref{th:nt} hold}.
\subsection{Basic results} \label{sec:PMF}
In this section we will establish some results about the majorant function $f:[0,\; R)\to \mathbb{R}$ and, some relationships between the majorant function and the set-valued mapping $F+T.$ For this, we begin by reminding that Proposition~3 of \cite{FerreiraSvaiter2009} state that the majorant function $f$ has a smallest root $t_*\in  (0,R)$, is strictly convex, $f(t)>0,$ $f'(t)<0$ and $t<t-f(t)/f'(t)< t_*$ for all $t\in [0,t_ *)$. Moreover, $f'(t_*)\leqslant 0$ and $f'(t_*)<0$ if, and only if, there exists $t\in (t_*,R)$ such that $f(t)< 0$. Let
$$
 \bar{t}:=\sup \left\{t\in [0,R): f'(t)<0 \right\}.
$$ 
Note that $t_*\leq \bar{t}$. Since $f'(t)<0$ for all $ t\in [0, \bar{t})$, the Newton iteration of the majorant function $f$ is well defined in $[0, \bar{t}).$ Let us call it $n_{f}: [0, \bar{t}) \to \mathbb{R}$ such that 
\begin{equation}\label{eq.majorfunc}
n_{f}(t)=t-\frac{f(t)}{f'(t)}.
\end{equation}
The next result will be used to obtain the convergence rate of the sequence generated by Newton's method for solving $f(t)=0.$ Its proof can be found in \cite[Proposition 4]{FerreiraSvaiter2009}.
\begin{lemma}\label{eq.ratemajor}
For all $t\in [0,t_*)$ we have $n_{f}(t)\in [0,t_*),$ $t<n_{f}(t)$ and $t_*-n_{f}(t)\leq \frac{1}{2}(t_*-t).$ If $f$ also satisfies the condition {\bf h4} then
$$
t_* -n_{f}(t) \leq \frac{D^{-}f'(t_*)}{-2f'(t_*)} (t_* -t)^2, \qquad \forall ~t\in [0,t_*).
$$
\end{lemma}
The definition on $\{t_k\}$ on \eqref{eq:DNS} is equivalent to the following one
\begin{equation}\label{eq.majseq}
t_0=0, \qquad t_{k+1}=n_{f}(t_k), \qquad k=0,1\ldots.
\end{equation}
The next result contain the main convergence properties of the above sequence and its prove, which is a consequence of Lemma~\ref{eq.ratemajor}, follows the same pattern as the proof of Corollary~2.15 of \cite{FerreiraMax2013}.
\begin{corollary}\label{major.convergence}
The sequence $\{t_k\}$ is well defined, strictly increasing and is contained in $[0,t_*).$ Moreover, it satisfies second inequality in \eqref{eq:rates0} and converges $Q$-linearly to $t_*.$ If also satisfies assumption {\bf h4}  then $\{t_k\}$ satisfies the second inequality in \eqref{ine.rates1} and converges $Q$-quadratically.
\end{corollary}
Therefore, we have obtained all the statements about the majorant sequence $\{t_k\}$ on Theorem~\ref{th:nt}.  Now we are going  to  establish  some relationships between the majorant function and the set-valued mapping $F+T.$ In the sequel we will  prove that the partial linearization of $F+T$ has a single-valued inverse, which is Lipschitz  in a  neighborhood of $x_0$. Since  Newton's iteration at a point in this neighborhood  happens to be a zero of the partial linearization of $F+T$ at such a point, it will be first  convenient to study the  {\it linearization error   of  $F$} at a point
in $\Omega$
\begin{equation}\label{eq:def.er}
  E_F(x,y):= F(y)-\left[ F(x)+F'(x)(y-x)\right],\qquad y,\, x\in \Omega.
\end{equation}
In the next result we  bound this error by the linearization error  of the majorant function $f$, namely, 
$$
 e_{f}(t,u):=f(u)-\left[f(t)+f'(t)(u-t)\right],\qquad t,\,u \in [0,R).
$$
\begin{lemma}  \label{pr:taylor}
Take $x,y\in B(x_0,R)$ and $0\leq t<v< R$. If $\|x-x_0\|\leq t$ and $\|y-x\|\leq v-t$ then
  \begin{equation}\label{eq:errormajor}
	\|\widehat{F'(x_0)}^{-1}\| \|E_F(x,y)\| \leq e_{f}(t,v)\frac{\|y-x\|^2}{(v-t)^2}.
	\end{equation}
\end{lemma}
\begin{proof}
Since $x+\tau(y-x)\in B(x_0,R),$ for all $\tau\in [0,1]$ and $F$ is continuously differentiable in $\Omega$, the linearization error of $F$ on \eqref{eq:def.er} is equivalent to 
$$
 E_F(x,y)=\int_{0}^{1} [F'(x+\tau(y-x))-F'(x)](y-x) d\tau, 
$$
which combined with the assumption in \eqref{Hyp:MH} and after some  simple algebraic manipulations we obtain 
$$
\|\widehat{F'(x_0)}^{-1}\| \|E_F(x,y)\| \leq \int_{0}^{1} [f'(\|x-x_0\| +\tau \|y-x\|) -f'(\|x-x_0\|)]\|y-x\| d\tau.
$$
Now, using the convexity of $f',$ the assumptions $\|x-x_0\|\leq t$ and $\|y-x\|<v-t,$ $v<R$  we have
\begin{eqnarray*}
   f'(\|x-x_0\| +\tau \|y-x\|) -f'(\|x-x_0\|)  &\leq& f'(t +\tau \|y-x\|) -f'(t) \\
	                                                    &\leq& [f'(t +\tau \|v-t\|) -f'(t)]\frac{\|y-x\|}{v-t}, 
\end{eqnarray*}
for any $\tau \in [0,1]$. Combining these inequalities we conclude that
	$$
	\|\widehat{F'(x_0)}^{-1}\| \|E_F(x,y)\| \leq \int_{0}^{1} [f'(t +\tau \|v-t\|) -f'(t)]\frac{\|y-x\|^2}{v-t}d\tau, 
	$$
which, after performing the integration we obtain  the desired result.
\end{proof}
In the next result  we will present the main relationships between the majorant function $f$ and the operator $F$. The  result is a consequence of Banach's lemma and its statement  is:
\begin{lemma} \label{le:wdns}
Let $x_0 \in \Omega$ be such that $\widehat{F'(x_0)}$ is a positive operator and $\widehat{F'(x_0)}^{-1}$ exists. If $\|x-x_0\|\leq t <t^*$, then $\widehat{F'(x)}$ is a positive operator and $\widehat{F'(x)}^{-1}$ exists. Moreover, 
$$
\|\widehat{F'(x)}^{-1}\|\leq  -\frac{\|\widehat{F'(x_0)}^{-1}\|}{f'(t)}.
$$
\end{lemma}
\begin{proof}
Firstly note that 
\begin{equation}\label{eq.matriz}
\|\widehat{F'(x)}-\widehat{F'(x_0)}\|\leq \frac{1}{2}\|F'(x)-F'(x_0)\| + \frac{1}{2}\|(F'(x)-F'(x_0))^*\|=\|F'(x)-F'(x_0)\|.
\end{equation}
Take $x\in B[x_0, t],$ $0\leq t<t^*$. Thus $f'(t)<0$. Using \eqref{eq.matriz}, \eqref{Hyp:MH} and taking into account {\bf h1} and {\bf h2} we obtain that
  \begin{equation}\label{eq:majcond}
    \|\widehat{F'(x_0)}^{-1}\| \|\widehat{F'(x)}-\widehat{F'(x_0)}\|\leq \|\widehat{F'(x_0)}^{-1}\| \|F'(x)-F'(x_0)\| \leq f'(\|x-x_0\|)-f'(0)<f'(t)+1< 1. 
  \end{equation}
	Thus, by Banach's lemma, we conclude that $\widehat{F'(x)}^{-1}$ exists. Moreover by above inequality,
	$$
	\|\widehat{F'(x)}^{-1}\|\leq \frac{\|\widehat{F'(x_0)}^{-1}\|}{1-\|\widehat{F'(x_0)}^{-1}\|\|F'(x)-F'(x_0)\|}\leq \frac{\|\widehat{F'(x_0)}^{-1}\|}{1-(f'(t)+1)} =-\frac{\|\widehat{F'(x_0)}^{-1}\|}{f'(t)}.
	$$
On the other hand, using \eqref{eq:majcond} we have
\begin{equation}\label{eq.selfadj}
\|\widehat{F'(x)}-\widehat{F'(x_0)}\|\leq \frac{1}{\|\widehat{F'(x_0)}^{-1}\|}.
\end{equation}
Take $y\in \banachh$. Then, it follows by above inequality that
$$
\langle (\widehat{F'(x_0)} -\widehat{F'(x)})y,y\rangle \leq \|\widehat{F'(x_0)} -\widehat{F'(x)}\|\|y\|^2\leq \frac{\|y\|^2}{\|\widehat{F'(x_0)}^{-1}\|},
$$
which implies, after of simple manipulations that
$$
\langle \widehat{F'(x_0)}y,y\rangle -\frac{\|y\|^2}{\|\widehat{F'(x_0)}^{-1}\|} \leq \langle \widehat{F'(x)}y,y\rangle.
$$
Since $\widehat{F'(x_0)}$ is a positive operator and $\widehat{F'(x_0)}^{-1}$ exists by assumption, we obtain by Lemma~\ref{eq:adjunt} that
$$
\langle \widehat{F'(x_0)}y,y\rangle \geq \frac{\|y\|^2}{\|\widehat{F'(x_0)}^{-1}\|}.
$$
Therefore, combining the two last inequalities we conclude that $\langle \widehat{F'(x)}y,y\rangle \geq 0$, i.e, $\widehat{F'(x)}$ is a positive operator.
\end{proof} 
Lemma~\ref{le:wdns} shows that $\widehat{F'(x)}$ is a positive operator and $\widehat{F'(x)}^{-1}$ exists, thus by Lemma~\ref{eq:adjunt} we have that for any $y\in \banachh$
$$
\langle \widehat{F'(x)}y,y\rangle \geq \frac{\|y\|^2}{\|\widehat{F'(x)}^{-1}\|}.
$$
Note that $\langle \widehat{F'(x)}y,y\rangle=\langle F'(x)y,y\rangle$, thus by the second part of Lemma~\ref{le:wdns} and $h_2$ we conclude that $F'(x)$ satisfies \eqref{eq.gooddef} and consequently, the Newton iteration mapping  is well-defined.  Let us call $N_{F+T}$, the Newton iteration mapping  for $F+T$ in that region, namely, $N_{F+T}:B(x_0, t^*) \to \banachh$ is defined by 
\begin{equation} \label{eq:NFef1}
N_{F+T}(x):= L_F(x,0)^{-1}.
\end{equation}
Using \eqref{eq:invplm}  we conclude that the definition of the Newton iteration mapping in \eqref{eq:NFef1} is equivalent to 
\begin{equation} \label{eq:NFef}
0\in L_F(x,N_{F+T}(x)):= F(x)+F'(x)(N_{F+T}(x)-x)+T(N_{F+T}(x)),\qquad \forall ~x\in  B(x_0, t^*).
\end{equation}
\begin{remark}\label{eq.initialpoint}
Since $0\in F(x_0)+F'(x_0)(x_1-x_0) +T(x_1)$, $x_1\in \Omega$, is single-valued, it follow from \eqref{eq:invplm} that
$$
\{x_1\}=L_F (x_0,0)^{-1}.
$$
\end{remark}

Therefore, one can apply a \emph{single} Newton iteration on any $x\in B(x_0, t^*)$ to obtain $N_{F+T}(x)$ which may not belong
to $B(x_0, t^*)$, or even may not belong to the domain of $F$. Thus, this is enough to guarantee the  well-definiteness of only one iteration of Newton's method. To ensure that Newtonian iterations may be repeated indefinitely or in particular, invariant on subsets of    $B(x_0, t_*)$,  we need some additional results. First, define some subsets of $B(x_0,t_*)$ in which, as we shall prove, Newton iteration mapping  \eqref{eq:NFef} are ``well behaved''. Define
\begin{equation}\label{eq:ker} 
{K}(t):=\left\{x\in \Omega ~: ~ \|x-x_0\| \leq t, \quad  \|L_F(x,0)^{-1}-x\| \leq -\frac{f(t)}{f'(t)}\right\}, \qquad t\in [0,t_*), 
\end{equation} 
\begin{equation} \label{eq:kt}  
{K}:=\bigcup_{t\in {[0,t_ *)}}
  K(t).
\end{equation}
\begin{lemma} \label{le:cl}
For each $0\leq t< t_*$ we have ${K}(t) \subset B(x_0,t_*)$ and 
	$
	N_{F+T}({K}(t)) \subset {K}(n_{f}(t)).
	$
	As a consequence, ${K}\subseteq B(x_0,t_ *)$ and $N_{F+T}({K}) \subset {K}$.
\end{lemma}
\begin{proof}
The first inclusion follows trivially from the definition of ${K}(t).$ Take   $x\in {K}(t).$ Using definition \eqref{eq:ker} and \eqref{eq.majorfunc} we concluded that
\begin{equation}\label{ini.cond}
\|x-x_0\| \leq t, \quad \qquad \|L_F(x,0)^{-1}-x\| \leq -\frac{f(t)}{f'(t)}, \qquad \quad  t< n_{f}(t) <t_ *.
\end{equation}
Definition of  Newton iteration mapping in \eqref{eq:NFef} implies that  for all $x  \in {K}(t)$ we have
$$
\|N_{F+T}(x)-x_0\| \leq \|x-x_0\| + \|N_{F+T}(x)-x\|= \|x-x_0\| + \|L_F(x,0)^{-1}-x\|, 
$$
and consequently,  using  \eqref{eq.majorfunc} and \eqref{ini.cond}, the last inequality imply 
\begin{equation} \label{eq:fcmt}
\|N_{F+T}(x)-x_0\| \leq t -\frac{f(t)}{f'(t)} = n_{f}(t) <t_ *.
\end{equation}
For simplify the notations define $x_+= N_{F+T} (x)$ and $y=L_F(x_+,0)^{-1}$. Thus, from \eqref{eq:NFef} we have
$$
0\in F(x)+F'(x)(x_+ -x)+T(x_+),\qquad 0\in F(x_+)+F'(x_+)(y -x_+)+T(y).
$$
As $T$ is a maximal monotone, it follows that
$$
\langle F(x)+F'(x)(x_+ -x) -F(x_+)-F'(x_+)(y -x_+), y-x_+ \rangle \geq 0
$$
which implies that
\begin{equation}\label{eq.monmax}
\langle F(x)-F(x_+)+F'(x)(x_+-x), y-x_+ \rangle \geq \langle F'(x_+)(y-x_+), y-x_+ \rangle.
\end{equation}
Since, by Lemma~\ref{le:wdns}, $\widehat{F'(x)}$ is a positive operator and $\widehat{F'(x)}^{-1}$ exists, we obtain from Lemma~\ref{eq:adjunt} that 
 \begin{equation}\label{eq.monmax1000}
\frac{\|y-x_+\|^2}{\|\widehat{F'(x_+)}^{-1}\|}\leq  \langle \widehat{F'(x_+)}(y-x_+),y-x_+\rangle.
\end{equation}
Note that
$
\langle \widehat{F'(x_+)}(y-x_+),y-x_+\rangle = \langle F'(x_+)(y-x_+), y-x_+\rangle,
$
this together \eqref{eq.monmax1000} and \eqref{eq.monmax} yields that
\begin{equation*}\label{eq.monmax2}
\|y-x_+\|^2 \leq \|\widehat{F'(x_+)}^{-1}\| \langle F'(x_+)(y-x_+),y-x_+\rangle \leq \|\widehat{F'(x_+)}^{-1}\| \langle F(x)-F(x_+)+F'(x)(x_+-x), y-x_+\rangle.
\end{equation*}
Hence, after simple manipulations, above inequality becomes
\begin{equation}\label{eq.monmax222}
\|y-x_+\| \leq \|\widehat{F'(x_+)}^{-1}\| \|F(x)-F(x_+)+F'(x)(x_+-x)\|.
\end{equation}
Due to  $x_+= N_{F+T} (x)$ we have from \eqref{eq:fcmt}  that $\|x_+-x_0\| \leq  n_{f}(t)$. Then, taking into account that  $f'$ is increasing and negative, it follows from \eqref{eq.monmax222} , using second part in Lemma~\ref{le:wdns}, \eqref{eq:def.er}  and Lemma~\ref{pr:taylor} we obtain that
\begin{equation}\label{eq.monmax3}
\|y-x_+\| \leq  -\frac{\|\widehat{F'(x_0)}^{-1}\|}{f'(n_{f}(t))}\|E_F(x,x_+)\| \leq \frac{-1}{f'(n_{f}(t))}e_{f}(t, n_{f}(t))\frac{\|x_+-x\|^2}{(n_{f}(t)-t)^2}.
\end{equation}
On the other hand, using   the definition  \eqref{eq.majorfunc}, after  some manipulations we conclude that
$$
f(n_{f}(t))= f(n_{f}(t)) -[f(t) + f'(t)(n_{f}(t)-t)]= e_{f}(t,n_{f}(t)), 
$$
 and because  $x_+= N_{F+T} (x)$,   \eqref{eq.majorfunc} and the  second inequality in \eqref{ini.cond}  imply  $\|x-x_+\|\leq n_{f}(t)-t$, thus inequality in \eqref{eq.monmax3} becomes
$$
\|y-x_+\|  \leq -\frac{f(n_{f}(t))}{f'(n_{f}(t))}.
$$
Therefore, since \eqref{eq:fcmt}  implies $\|x_+-x_0\| \leq n_{f}(t)$  we conclude that the second inclusion of the proposition  is proved.

The third  inclusion ${K}\subseteq B(x_0,t_*)$ follows trivially from \eqref{eq:ker} and \eqref{eq:kt}. To prove the last inclusion $N_{F+T}({K}) \subset {K}$, take $x\in {K}$. Thus $x\in K(t)$ for some $t\in [0,t_ *)$. From the second inclusion  of the proposition, we have $N_{F+T}(x) \in {K}(n_{f}(t))$. Since $n_{f}(t)\in [0,t_ *)$ and using the definition of ${K}$ in \eqref{eq:kt} we concluded the proof.
\end{proof}

\subsection{Convergence analysis}\label{convanalysis}
To prove the convergence result, which is a consequence of the above results, firstly we note that the definition \eqref{eq:NFef1} implies that the sequence $\{x_k\}$ defined in \eqref{eq:DNS},  can be  formally stated by 
\begin{equation}\label{eq.seq}
x_{k+1}=N_{F+T}(x_k), \qquad k=0,1,\ldots,
\end{equation}
or equivalently, 
$$
 0\in F(x_k)+F'(x_k)(x_{k+1}-x_k)+T(x_{k+1}),  \qquad k=0,1, \ldots.
$$
First we will  show that the sequence generated by Newton method is well behaved with respect to the set defined in \eqref{eq:ker}.
\begin{corollary}\label{res.solution}
The sequence $\{x_k\}$ is well defined, is contained in $B(x_0,t_ *),$ converges to a point $x_*\in B[x_0,t_ *]$  satisfying  $0\in F(x_*)+T(x_*).$ Moreover,   $x_k\in {K}(t_k)$, for $k=0,1\ldots$ and 
$$
\|x_*-x_k\|\leq t_ *-t_k, \qquad k=0,1\ldots.
$$
\end{corollary}
\begin{proof}
From assumption in \eqref{eq.ipoint1}, Remark~\ref{eq.initialpoint},  assumption ${\bf h1}$ and definitions \eqref{eq:ker}  and \eqref{eq:kt}   we have
\begin{equation} \label{eq:fs}
\{x_0\}={K}(0)\subset {K}.
\end{equation}
We know from Lemma~\ref{le:cl} that $N_{F+T}({K}) \subset {K}$. Thus, using \eqref{eq:fs} and \eqref{eq.seq}  we conclude that the sequence $\{x_k\}$ is well defined and rests in ${K}.$  From the first inclusion on second part of the Lemma~\ref{le:cl} we have trivially that $\{x_k\}$ is contained in $B(x_0,t_ *).$  To prove the convergence,  first we are going to prove by induction  that 
\begin{equation}\label{eq.defseq}
x_k\in {K}(t_k), \qquad k=0,1\ldots.
\end{equation}
The above inclusion, for $k=0$,  follows from \eqref{eq:fs}. Assume now that $x_k\in {K}(t_k).$ Then combining Lemma~\ref{le:cl}, \eqref{eq.seq} and \eqref{eq.majorfunc} we conclude that $x_{k+1}\in {K}(t_{k+1}),$ which completes the induction proof. Now,  using \eqref{eq.defseq} and \eqref{eq:ker} we have 
$$
\|L_F(x_k,0)^{-1}-x_k\| \leq -\frac{f(t_k)}{f'(t_k)}, \qquad k=0,1 \ldots,
$$
which, combined with  \eqref{eq.seq} and \eqref{eq:DNS}  becomes 
\begin{equation}\label{des.conver}
\|x_{k+1}-x_k\|\leq t_{k+1}-t_k, \qquad k=0,1 \ldots.
\end{equation}
Taking into account that  $\{t_k\}$ converges to $t_ *,$   we  easily  conclude from  the above inequality  that
$$
\sum_{k=k_0}^{\infty} \|x_{k+1}-x_k\| \leq \sum_{k=k_0}^{\infty} t_{k+1}-t_k =t_ *- t_{k_0} < +\infty,
$$
for any $k_0 \in \mathbb{N}.$ Hence, we conclude that $\{x_k\}$ is a Cauchy sequence in $B(x_0, t_ *)$ and thus it  converges to some $x_* \in B[x_0, t_ *].$  Therefore, using  again \eqref{des.conver} we also conclude  that the inequality in the corollary holds.  

Since $-F(x_k)-F'(x_k)(x_{k+1}-x_k)\in T(x_{k+1})$, $F$ is a continuously differentiable mapping, $x_k$ converges to $x_*$ and $T$ is maximal monotone, we conclude that $0 \in F(x_*)+T(x_*)$.
\end{proof}

We have already proved that   the sequence $\{x_k\}$ converges to a solution $x_*$ of generalized equation $F(x)+T(x)\ni 0$ and $x_*\in B(x_0,t_*)$. Now, we will prove that this convergence is $Q$-linearly and that $x^*$ is the unique solution of $F(x)+T(x)\ni 0$ in $B(x_0,t_*)$. Furthermore, by assuming that $f$ satisfies ${\bf h4}$, we will also prove that $\{x_k\}$ converges $Q$-quadratically to $x_*$.  For that, we need of the following result: 
\begin{lemma}\label{ine.rates}
Take $x,y\in B(x_0,R)$ and $0\leq t<v<R$. If
\begin{equation} \label{eq:siqc}
t\leq t^*,\quad \|x-x_0\|\leq t, \quad \|y-x\|\leq v-t, \quad f(v)\leq 0,  \quad 0\in F(y)+T(y), 
\end{equation}
then the following inequality holds
$$
\|y-N_{F+T}(x)\|\leq [v-n_{f}(t)]\frac{\|y-x\|^2}{(v-t)^2}.
$$
\end{lemma}
\begin{proof}
For simplify the notations define $z= N_{F+T} (x)$. Since $0\in F(y)+T(y)$ using \eqref{eq:NFef} and that $T$ is maximal monotone we have
$$
\langle -F(x)+F'(x)(x-z)+F(y), z-y\rangle \geq 0, 
$$
which, after simple manipulations, implies that
\begin{equation}\label{eq.monmax1}
\langle F(y)-F(x)+F'(x)(x-y), z-y\rangle \geq \langle F'(x)(z-y),z-y\rangle.
\end{equation}
Since $\|x-x_0\|\leq t<t^*$ we obtain by Lemma~\ref{le:wdns} that $\widehat{F'(x)}$ is a positive operator and $\widehat{F'(x)}^{-1}$ exists. Thus from Lemma~\ref{eq:adjunt} we have 
 \begin{equation}\label{eq.max}
\frac{\|z-y\|^2}{\|\widehat{F'(x)}^{-1}\|}\leq  \langle \widehat{F'(x)}(z-y),z-y\rangle.
\end{equation}
Combining $\langle \widehat{F'(x)}(z-y),z-y\rangle = \langle F'(x)(z-y), z-y\rangle,$ with \eqref{eq.max} and \eqref{eq.monmax1} yields that
\begin{equation*}\label{eq.monmax2}
\|y-z\|^2 \leq \|\widehat{F'(x)}^{-1}\| \langle F'(x)(z-y),z-y\rangle \leq \|\widehat{F'(x)}^{-1}\| \langle F(y)-F(x)+F'(x)(x-y), z-y\rangle.
\end{equation*}
Hence, after simple manipulations, above inequality becomes
\begin{equation}\label{eq.max20}
\|y-z\| \leq \|\widehat{F'(x)}^{-1}\| \|F(y)-F(x)+F'(x)(x-y)\|.
\end{equation}
Now, using Lemma~\ref{le:wdns} and Lemma~\ref{pr:taylor} together with the assumptions of the lemma we obtain 
$$
\|y-z\| \leq \frac{-1}{f'(t)}e_{f}(t, v)\frac{\|y-x\|^2}{(v-t)^2}.
$$
As, $0\leq t<t_ *,$ $f'(t)<0$. Using definition of $e_{f}(t, v)$, \eqref{eq.majorfunc} and the assumption $f(v)\leq 0$ we have
$$
-\frac{e_{f}(t,v)}{f'(t)}=v-t + \frac{f(t)}{f'(t)}-\frac{f(v)}{f'(t)} \leq v-t + \frac{f(t)}{f'(t)}=v-n_{f}(t).
$$
To end the proof, combine the two above inequalities. 
\end{proof}

\begin{corollary}
The sequences  $\{x_k\}$ and $\{t_k\}$ satisfy  the following inequality
\begin{equation}\label{ine.quadr}
\|x_*-x_{k+1}\|\leq \frac{t_*-t_{k+1}}{(t_* -t_k)^2}\|x_*-x_k\|^2, \qquad k=0,1\ldots.
\end{equation}
As a consequence,  the  sequence  $\{x_k\}$ converges  $Q$-linearly to the solution  $x^*$ as follows
\begin{equation}\label{ine.quadr1}
\|x_*-x_{k+1}\|\leq \frac{1}{2} \|x_*-x_k\|, \qquad k=0,1\ldots.
\end{equation}
Additionally, if $f$ satisfies ${\bf h4}$ then the  sequence $\{x_k\}$ converges $Q$-quadratically to $x_*$  as follows
\begin{equation}\label{ine.quadr2}
\|x_*-x_{k+1}\| \leq \frac{D^{-}f'(t_*)}{-2f'(t_*)}\|x_*-x_k\|^2, \qquad k=0,1\ldots.
\end{equation}
\end{corollary}
\begin{proof}
For each  $k$,  we can apply  Lemma~\ref{ine.rates} with $x=x_k,$ $y=x_*,$ $t=t_k$ and $v=t_*,$ to obtain
$$
\|x_*-N_{F+T}(x_k)\|\leq [t_*-n_{f}(t_k)]\frac{\|x_*-x_k\|^2}{(t_*-t_k)^2}.
$$
Thus  inequality \eqref{ine.quadr} follows from the above inequality, \eqref{eq.seq} and \eqref{eq.majseq}. Note that by the first part in Lemma~\ref{eq.ratemajor}, \eqref{eq.majseq} and  Corollary~\ref{res.solution} we have 
$$
\frac{t_*-t_{k+1}}{t_*-t_k}\leq \frac{1}{2}, \qquad  \qquad  \frac{\|x_*-x_k\|}{t_*-t_k}\leq 1.
$$
Combining these inequalities with \eqref{ine.quadr} we obtain \eqref{ine.quadr1}.
Now, assume that ${\bf h4}$ holds. Then, by Corollary~\ref{major.convergence}, the second inequality on \eqref{ine.rates1} holds, which combined with \eqref{ine.quadr} imply \eqref{ine.quadr2}. 
\end{proof}

\begin{corollary}
The limit $x_*$ of the sequence $\{x_k\}$ is the unique solution of the generalized equation $F(x)+T(x)\ni 0$ in $B[x_0, t_ *]$.
\end{corollary}
\begin{proof}
Suppose there exist $y_* \in B[x_0, t_ *]$ such that $y_*$ is solution of $F(x)+T(x)\ni 0$. We will prove by induction that
\begin{equation}\label{iq.indu}
\|y_*-x_k\|\leq t_*-t_k, \qquad  k=0,1,\ldots.
\end{equation}
The case $k=0$ is trivial, because $t_0=0$ and $y_* \in  B[x_0,t_*]$. We assume that the inequality holds for some $k$.  First note that Corollary~\ref{res.solution} implies that  $x_k\in {K}(t_k)$, for $k=0,1\ldots$. Thus, from definition of ${K}(t_k)$  we conclude that   $\|x_k-x_0\|\leq t_k$, for $k=0,1\ldots$. Since $\|x_k-x_0\|\leq t_k$, we may apply Lemma~\ref{ine.rates} with $x=x_k$, $y=y_*$, $t=t_k$ and $v=t_*$ to obtain
$$
\|y_*-N_{F+T}(x_k)\|\leq [t_*-n_{f}(t_k)]\frac{\|y_*-x_k\|^2}{(t_*-t_k)^2}.
$$ 
Using inductive hypothesis, \eqref{eq.seq} and \eqref{eq.majseq} we obtain, from latter inequality, that \eqref{iq.indu} holds for $k+1$. Since $x_k$ converges to $x_*$ and $t_k$ converges to $t_*$, from \eqref{iq.indu} we conclude that $y_*=x_*$. Therefore, $x_*$ is the unique solution of $F(x)+T(x)\ni 0$ in $B[x_0,t_*]$.
\end{proof}
\section{Some special cases} \label{apl}
In this section, we will present some  special cases of  Theorem \ref{th:nt}. When    $F\equiv \{0\}$  and    $f'$ satisfies a Lipschitz-type condition,  we will obtain a particular instance of Theorem~\ref{th:nt}, which retrieves  the classical convergence theorem on Newton's method under the Lipschitz condition; see \cite{Rall1974, Traub1979}.  A version of Smale's theorem on Newton's method for analytical functions is obtained in Theorem~\ref{theo:Smale}.
\subsection{Under Lipschitz-type condition}
In this section,  we will present a version of  classical convergence theorem for Newton's method under Lipschitz-type condition  for generalized equations. The classical version for  $F\equiv \{0\}$   have  appeared   in  Rall \cite{Rall1974} and Traub and  Wozniakowski \cite{Traub1979}.
\begin{theorem} \label{th:cc}
Let $\banachh$ be a Hilbert space, $\Omega$ be an open nonempty subset of $\banachh$, $F: \Omega \to \banachh$ be continuous with Fr\'echet derivative $F'$ continuous, $T:\banachh \rightrightarrows  \banachh$ be a set-valued operator and $x_0\in \Omega$. Suppose that $F'(x_0)$ is a positive operator and $\widehat{F'(x_0)}^{-1}$ exists and, there exists a constant $K>0$ such that  $B(x_0, 1/K)\subset \Omega$ and
\begin{equation} \label{eq:hc}
\|\widehat{F'(x_0)}^{-1}\| \|f'(x)-f'(y)\| \leq K \|x-y\|,\qquad x,\, y\in B(x_0, 1/K).
\end{equation}
 Moreover, suppose that
	\begin{equation} \label{eq.ipoint}
	\|x_1-x_0\|\leq b.
	\end{equation}
Then, the sequence $\{x_k\}$ generated by Newton's method for solving $F(x)+T(x)\ni 0$ with starting point $x_0$
\begin{equation} \label{eq:Kant}
 F(x_k)+F'(x_k)(x_{k+1}-x_k)+T(x_{k+1})\ni 0, \qquad k=0,1,\ldots\,, 
\end{equation}
is well defined,  is contained in $B(x_0, t^*)$ and  converges to the point $x^*$ which is the unique solution of $F(x)+T(x)\ni 0$ in $B(x_0, t^*)$, where $t^*= 1-\sqrt{1-2bK}/{K}$. Moreover, the sequence $\{x_k\}$ satisfies for  any $k=0,1,\ldots,$
$$
 \|x_*-x_{k+1}\| \leq \frac{K}{2\sqrt{1-2bK}}\|x_*-x_k\|^2.
$$
\end{theorem}
\begin{proof} 
Since $f:[0,1/K)\to \mathbb{R},$  defined by $f(t):=(K/2)t^2-t+b,$ is a majorant function for $F$ at point $x_0$, the result follows by invoking Theorem~\ref{th:nt}, applied to this particular context.
\end{proof}
\begin{remark}
The  above result  contain,  as particular instance,  several   theorem on Newton's method; see,  for example, Rall \cite{Rall1974},  Traub and  Wozniakowski \cite{Traub1979} and Daniel \cite{Daniel1973}. See also \cite{Wang2015_1}.
\end{remark}

\subsection{Under Smale's-type condition}
In this section,  we will present a version of  classical convergence theorem for Newton's method under Smale's-type condition for generalized equations. The classical version  has  appeared   in   corollary of Proposition 3 pp.~195 of Smale \cite{Smale1986}, see also Proposition 1 pp.~157 and Remark 1 pp.~158 of  Blum, Cucker,  Shub, and Smale~\cite{BlumSmale1998}; see also \cite{Ferreira2009}.

\begin{theorem} \label{theo:Smale}
Let $\banachh$ be a Hilbert space, $\Omega$ be an open nonempty subset of $\banachh$, $F: \Omega \to \banachh$ be an analytic function, $T:\banachh \rightrightarrows  \banachh$ be a set-valued operator and $x_0\in \Omega$. Suppose that $F'(x_0)$ is a positive operator and $\widehat{F'(x_0)}^{-1}$ exists. Suppose that   
\begin{equation} \label{eq:SmaleCond}
   \gamma:= \|\widehat{F'(x_0)}^{-1}\|\sup _{ n > 1 }\left\| \frac  {F^{(n)}(x_0)}{n!}\right\|^{1/(n-1)}<+\infty.
\end{equation}
 Moreover, suppose that $B(x_0,1/\gamma)$ and there exists $b>0$ such that
	\begin{equation} \label{eq.ipoint}
	\|x_1-x_0\|\leq b
	\end{equation}
and $\alpha := b \gamma\leq 3-2\sqrt{2}$. Then the sequence $\{x_k\}$ generated by Newton's method for solving $F(x)+T(x)\ni 0$ with starting point $x_0$
\begin{equation} \label{eq:Kant}
 F(x_k)+F'(x_k)(x_{k+1}-x_k)+T(x_{k+1})\ni 0, \qquad k=0,1,\ldots\,, 
\end{equation}
is well defined,  is contained in $B(x_0, t^*)$ and  converges to the point $x^*$ which is the unique solution of $F(x)+T(x)\ni 0$ in $B[x_0, t^*]$, where $t_ *=(\alpha +1-\sqrt{(\alpha+1)^2 -8\alpha})/4\gamma$.
Moreover, $\{x_k\}$ converges $Q$-linearly as follows
  \[
 \|x_*-x_{k+1}\| \leq \frac{1}{2}\|x_* -x_k\|, \qquad k=0,1,\ldots.
  \]
Additionally, if $\alpha < 3-2\sqrt{2}$, then $\{x_k\}$ converges $Q$-quadratically as follows
	$$
	\|x_*-x_{k+1}\| \leq \frac{\gamma}{(1-\gamma t_ *)[2(1-\gamma t_ *)^2-1]}\|x_*-x_k\|^2,\qquad k=0,1,\ldots.
	$$
\end{theorem}
Before proving above theorem we need of two results. The next results gives a condition  that is easier to check than condition
\eqref{Hyp:MH}, whenever  the functions under consideration  are twice continuously differentiable, and its proof follows the same path of Lemma~21 of \cite{FerreiraGoncalvesOliveira2011}. 
\begin{lemma}\label{lem.cond1}
Let $\Omega \subset \banachh$ be an open set, and let $F:{\Omega}\to \banachh$ be an analytic function. Suppose that $x_0 \in \Omega$ and $B(x_0, 1/ \gamma)\subset \Omega,$ where $\gamma$ is defined in \eqref{eq:SmaleCond}. Then for all $x\in B(x_0, 1/  \gamma),$ it holds that
$
\|F''(x)\|\leq 2  \gamma/(1-  \gamma\|x-x_0\|)^3.
$
\end{lemma}
The next result  gives a relationship  between  the second derivatives $F''$ and $ f''$,  which allow us to show that  $F$ and $f$ satisfy \eqref{Hyp:MH}, and its proof  is similar to Lemma~22 of \cite{FerreiraGoncalvesOliveira2011}. 
\begin{lemma} \label{lc}
Let $\banachh$ be a Hilbert space, $\Omega\subseteq \banachh$ be an open set,
  $F:{\Omega}\to \banachh$  be twice continuously differentiable. Let $x_0 \in \Omega$, $R>0$  and  $\kappa=\sup\{t\in [0, R): B(x_0, t)\subset \Omega\}$. Let \mbox{$f:[0,R)\to \mathbb {R}$} be twice continuously differentiable such that $ \|\widehat{F'(x_0)}^{-1}\|\|F''(x)\|\leqslant f''(\|x-x_0\|),$
for all $x\in B(x_0, \kappa)$, then $F$ and $f$ satisfy \eqref{Hyp:MH}.
\end{lemma}

\noindent
{\bf [Proof of Theorem \ref{theo:Smale}]}.
Consider $f:[0, 1/ \gamma) \to \mathbb{R}$ defined by $f(t)=t/(1- \gamma t)-2t+b$. Note that $f$ is  analytic and 
$f(0)=b$,   $f'(t)=1/(1- \gamma t)^2-2$, $f'(0)=-1$, $f''(t)=2 \gamma/(1-\gamma t)^3$. 
It follows from the last  equalities  that $f$ satisfies {\bf h1}  and  {\bf h2}.  Combining  Lemma~\ref{lc}  with  Lemma~\ref{lem.cond1}, we conclude  that $F$  and $f$ satisfy  \eqref{Hyp:MH}. Therefore, the result follows by applying the Theorem~\ref{th:nt}.
\qed
\section{Final remarks } \label{rf}
We have obtained  a semi local  convergence result to  Newton's method for solving  generalized equation in Hilbert spaces and  under the majorant condition. The majorant condition  allow to unify several  convergence results pertaining to  Newton's method. Besides, the study of  inexact versions  of this method would be welcome.

\end{document}